\documentclass[a4paper,oneside,12pt]{amsart}

\usepackage[utf8]{inputenc}
\usepackage{latexsym}
\usepackage{amssymb}
\usepackage{amsmath}
\usepackage{stmaryrd} 

\usepackage{enumerate} 

\usepackage{hyperref}

\usepackage{pgf}

\usepackage{natbib}

\usepackage{mathtools,amsthm}
\usepackage{tikz}
\usetikzlibrary{decorations.pathreplacing}

\usepackage{ifpdf}
\usetikzlibrary{calc}
\usepackage{fp} 

\bibpunct[, ]{(}{)}{,}{a}{}{;}

\def\qed{\relax
   \ifmmode
    ~\hfill\Box
   \else
    \unskip\nobreak ~\hfill$\square$%
   \fi \par}

\newcommand{\sep}{$\cdot$ }
\def\keywords{\vspace{.5em}
{\noindent\textbf{Keywords}:\,\relax%
}}

\theoremstyle{definition} \newtheorem{cor}{Corollary}
\theoremstyle{definition} \newtheorem{conv}{Convention}
\theoremstyle{definition} \newtheorem{prop}{Proposition}
\theoremstyle{definition} 
\theoremstyle{definition} \newtheorem{rem}{Remark}
\theoremstyle{definition} \newtheorem{lemma}{Lemma}

\newcommand{\MM}{\mathfrak{M}} 
\newcommand{\mng}[2]{\left\llbracket#1\right\rrbracket^{#2}} 

\newcommand{\ie}{i.e., }
\newcommand{\eg}{e.g., }

\newcommand\defeq{\mathrel{\stackrel{\mathclap{\textsf{\tiny def}}}{=}}}

\makeatletter
\newcommand{\BIG}{\bBigg@{2}}
\newcommand{\BIGG}{\bBigg@{3}}
\newcommand{\vast}{\bBigg@{4}}
\newcommand{\Vast}{\bBigg@{5}}
\makeatother

\newcommand{\LL}{\mathcal{L}}

\newcommand{\Pred}{\ensuremath{\mathsf{Pred}}}
\newcommand{\Fm}{\ensuremath{\mathsf{Form}}}
\newcommand{\Var}{\ensuremath{\mathsf{Var}}}
\newcommand{\ar}{\ensuremath{\mathsf{ar}}}
\newcommand{\var}{\ensuremath{\mathsf{v}}}
\newcommand{\vi}{\ensuremath{\mathsf{v}_{i}}}
\newcommand{\vj}{\ensuremath{\mathsf{v}_{j}}}

\definecolor{axcolor}{rgb}{.3,0,.3}

\newcommand{\bexists}[2]{(\exists #1 \in #2 )} 
\newcommand{\bforall}[2]{(\forall #1 \in #2 )} 
 
\newcommand{\Bforall}[2]{\big(\forall #1 \in #2 \big)}

\newcommand{\AND}{\land}

\newcommand{\de}{\stackrel{\text{\tiny def}}{=}}
\newcommand{\defiff}{\ \stackrel{\text{\tiny{def}}}{\Longleftrightarrow}\ }

\usepackage{xspace}

\newcommand{\ax}[1]{{\ensuremath{\mathsf{#1}}}}
\newcommand{\sy}[1]{{\ensuremath{\mathsf{#1}}}}

\theoremstyle{definition} \newtheorem{thm}{Theorem}
\theoremstyle{definition} 
\theoremstyle{definition} \newtheorem{defn}{Definition}
 
\newcommand{\IOb}{\ensuremath{\mathit{IOb}}} 

\newcommand{\Q}{\ensuremath{\mathit{Q}}} 
\newcommand{\W}{\ensuremath{\mathit{W}}} 

\newcommand{\speed}{\mathit{speed}}
\newcommand{\Ether}{\ensuremath{\mathit{Ether}}\xspace}

\begin{document}
	\title{{\tiny -- Preprint --}\\ On variable
            non-dependence of first-order formulas}
        \author{Koen Lefever \and Gergely Sz{\'e}kely} \date{\today}
        \maketitle

\begin{abstract}
In this paper, we introduce a concept of non-dependence of variables in formulas. A formula in first-order logic is non-dependent of a variable if the truth value of this formula does not depend on the value of that variable. This variable non-dependence can be subject to constraints on the value of some variables which appear in the formula, these constraints are expressed by another first-order formula. After investigating its basic properties, we apply this concept to simplify convoluted formulas by bringing out and discarding redundant nested quantifiers. Such convoluted formulas typically appear when one uses a translation function interpreting a theory into another.
\end{abstract}

\keywords First-Order Logic \sep Algebraic Logic \sep Model Theory \sep Cylindric Algebras
\sep Simplification Rules \sep Translation Functions \sep Logical Interpretation \sep Nested Quantifiers

\section{Introduction}
\label{intro}

In general, it is not possible to bring out and discard nested quantifiers from formulas in first-order logic. In this paper, we will however present some cases in which this is possible. In order to do so, we introduce the notion of variable non-dependent\footnote{We use the term \textit{non-dependent} to avoid confusion with other usages of the term \textit{independent} in logic and with the term \textit{independent variable} which in mathematics is used for a symbol that represents an arbitrary value in the domain of a function, see, e.g., \citep[Section~1.1]{Stewart}.} formulas.

We are going to call a formula $\varphi$ \emph{non-dependent of
variable $x$} if the truth or falsity of formula $\varphi$ does not
depend on how variable $x$ is interpreted, \ie which value we assign
to $x$. To achieve non-dependence, we may need to put restrictions on
the scope of interpretation of $x$ and that of other variables.  So in
general, we say that $\varphi$ is \emph{non-dependent of variable $x$
in a model provided some condition }captured by another
formula $\theta$, for a precise definition, see
Definition~\ref{def:indep-prov} on p.\pageref{def:indep-prov}.

There are various ways in which a formula can be non-dependent of variable $x$:\footnote{While the examples here are from mathematics and assume that the variables are numbers, we do not make that assumption on the nature of the variables in our definitions and theorems below: ``$x$ is human'' is dependent of $x$; ``$k$ is an inertial observer according to observer $x$'' is non-dependent of $x$ (in classical and relativistic kinematics).}
\begin{itemize}
\item The formula does not contain $x$, \eg $1 \leq y \leq 2$ as illustrated\footnote{In Figure~\ref{fig-non-dependence} we present the main concepts and ideas in a naive and intuitive way, simplified to two numerical dimensions. In following figures, we will use our formal framework more rigorously and also allow infinitely many variables of any kind.} on the right in Figure~\ref{fig-non-dependence} is non-dependent of $x$ in every model for any language containing binary predicate $\le$. 
\item The formula contains $x$, but $x$ is bounded (\ie it does not occur free) in the formula, \eg $\exists x (y\neq x)$ is non-dependent of $x$ in every model.
\item The formula contains $x$, but is always true or always false, \eg $\exists y (y\neq x)$ is non-dependent of $x$ in every model (it is always true if the model has at least two elements and false otherwise).
\item The formula contains $x$ and is not always true or false, but is non-dependent of the value of $x$, \eg $(x^2 + 1)(y^2 - y) > 0$ is non-dependent of variable $x$ in the ordered field of real numbers.
\item The formula is non-dependent of $x$ provided some condition, \eg $x (y^2 - y) \geq 0$ is non-dependent of variable $x$ in the ordered field of real numbers provided $x$ is positive.
\end{itemize}

\begin{figure}
  \begin{center}
    \includegraphics{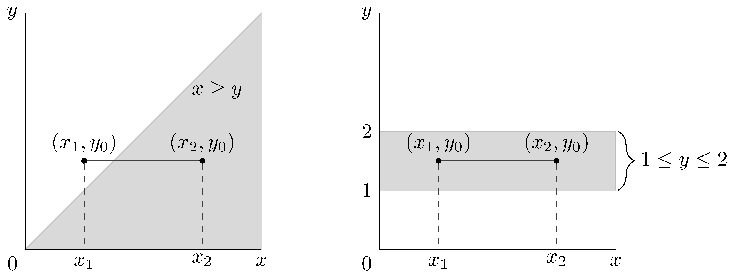}
  \caption{\label{fig-non-dependence}\footnotesize{Here the grey color represents values which make the formulas true and white represents values which make the formulas false. \\ On the left we have the formula $x \geq y$ of which the truth value is dependent of both the variables $x$ and $y$. Only changing the value of $x$ while keeping $y$ constant can change the truth value of this statement. \\ On the right we have the formula $1 \leq y \leq 2$ which is variable non-dependent of $x$. Whatever value we choose for $x$, the truth value of that statement does not change since it is only dependent of the value of $y$.}}
  \end{center}
\end{figure}

In general, mathematical theorems can be viewed as special cases of variable non-dependence. For example, by Fermat's little theorem,\footnote{See Pierre de Fermat's letter to Fr{\'e}nicle de Bessy, dated October 18, 1640 in \citep[pp.~206-212]{Fermat-corresp}.} formula $\exists x (x\cdot p=a^p-a)$ is non-dependent of variable $p$ (in the ring of integers) provided that $p$ is a prime number.

Mathematical translation functions, which accept a well-formed formula in one formal language and mechanographically transform it into a formula in another formal language, rarely produce esthetically pleasing results. This is partly due to the fact that they have to add conditions to the formula which take the constraints of the theories behind the languages between which is being translated into account. 

For example, in \citep{diss} and \citep{ClassRelKin},\footnote{Our
work is part of a broader tradition of using methods from mathematical
logic to compare scientific theories in general, and relativity
theories in particular. See, e.g., \citep{BigBook}, \citep{Ma10},
\citep{Szabo}, \citep{StannettNemeti}, \citep{Friend15},
\citep{Govindarajalulu}, \citep{Laurenz}, \citep{8156948},
\citep{weatherall}, \citep{Morita}, \citep{Luo_2016},
\citep{DBLP:journals/corr/abs-1912-07060}, \citep{Halvorson_2019},
\citep{Khaled19}, \citep{ANDREKA2021}, \citep{Formica2021},
\citep{10.1007/978-3-030-66501-2_48}, \citep{Weatherall2018WhyNC},
\citep{10.1215/00294527-2022-0029}, \citep{MADARASZ2022103153},
\citep{MEADOWS_2023}, \citep{Weatherall2024}, \citep{JPVBLev},
\citep{Enayat}, \citep{KHALED_SZEKELY_2024}, and \citep{aslan2024}.}
an axiom system for special relativity was interpreted\footnote{For a
discussion on the relation between translations, interpretations and
definitional equivalence, see e.g., \citep{HMT71}, \citep{pinter},
\citep{Vis06}, \citep{definability}, \citep{Glymour},
\citep{NonDisjointLanguages}, or \citep{MCELDOWNEY2020}.} into the
language of late classical kinematics by a translation function. The
translation function has to add the condition to each inertial
observer that they have to go slower than light, which results in
convoluted nested formulas if the original formula includes multiple
inertial observers. This condition was expressed in the ether frame of
refence. To simplify translations, since all observers representing
the ether frame are at rest relative to each other, we were allowed to
assume that all inertial observers chose the same ether-representing
observer when the formula was built up from relations whose meaning
was non-dependent of the choice of this ether-representing
observer. For example, that ``the speed of something is $v$ according
to the ether-representing observer'' is not dependent of the
ether-representing observer, but that ``the speed of the
ether-representing observer is $v$ according to some other observer''
is not.

Let us consider as an illustration the axiom \ax{AxSelf}, which states that every  inertial observer is stationary in its own coordinate system:\footnote{See, e.g., \citep[p.~160]{Andreka2006}. In this axiom, $\IOb$ is the set of inertial observers, $Q$ is the set of quantities (where $\langle \Q,+,\cdot,\le \rangle$ is an Euclidean Field), and $\W$ is the \emph{worldview relation} capturing coordinatization. The axiom intuitively says that all inertial observers measure their own postion relative to themselves at coordinates $(t,0,0,0)$ at any time $t$.}
\begin{equation*}
\bforall {k}  {IOb}
\bforall {t,x,y,z}{\Q}
\big[\W(k,k,t,x,y,z) \leftrightarrow x=y=z=0\big].
\end{equation*}
If we translate this axiom from special relativity to classical kinematics we get\footnote{See \citep{diss}: p.~12 for the definition of the speed of observer $k$ relative to the ether $\speed_{e}(k)$, p.~19 for the definition of the set of all ether observers $\Ether$,  p.~30 for the definition of the radarization function $Rad_{\bar v}$ (this is used to transform between classical and relativistic co-ordinates: it is in essence a Galilean transformation followed by a Lorentz transformation, its inverse $Rad^{-1}_{\bar v}$ is a Lorentz transformation followed by a Galilean transformation), pp.~33-35 for the definition of the translation function, p.~35 for a discussion on the translation of the speed of light $c$, and p.~78 for a discussion on the simplification of the translated axiom \ax{AxSelf}.}
\begin{multline*}
\bforall {k}  {IOb} 
\underline{\bforall {e} {\Ether}}\Big( \speed_{e}(k) < c \\\to
\Bforall {t,x,y,z}{\Q} \underline{\bforall {e} {\Ether}}  \big[ W\big(k,k,Rad^{-1}_{\bar{v}_k(e)}(k,k,t,x,y,z)\big) \leftrightarrow x=y=z=0 \big] \Big).
\end{multline*}

Note that $\bforall {e} {\Ether}$ occurs twice\footnote{The first is generated by the translation of $IOb$, the second is generated by the translation of $\W$.} in the translated formula. With the methods developed in \citep[\S~11~Appendix]{diss} and with the more generic method we present in the current paper\footnote{See Theorem~\ref{prop-simp} in Section~\ref{sec:Application} below.} we can simplify this to
\begin{multline*}
\bforall {k} { IOb}  \underline{\bforall {e} {\Ether}}  
\Big( \speed_{e}(k) < c \\\to
\Bforall {t,x,y,z}{\Q} \big[ W\big(k,k,Rad^{-1}_{\bar{v}_k(e)}(k,k,t,x,y,z)\big) \leftrightarrow x=y=z=0 \big] \Big)
\end{multline*}
because the statement does not depend on which ether observer $e$ is chosen. This simplified translation is a lot easier to prove\footnote{See \citep[p.~39]{diss} for a proof (using the simplification from this example) that \ax{AxSelf} translated from special relativity to classical kinematics is a theorem in classical kinematics, which is one of the steps in showing that the given translation is an interpretation.} than the original mechanographical translation containing redundant nested quantifiers.\\

This idea of variable non-dependence naturally appears in certain formalizations of Einstein's Special Principle of Relativity, see \citep[\S2.8.3]{Judit} and \citep{DiffFormRelPrinc}. 
Using their formal language, the formalizations there can be reformulated in terms of variable
non-dependence because their formulation intuitively says that the truth or falsity of a formal description $\varphi(k,\bar x)$ of a physical experiment is non-dependent of variable $k$ provided $k$ is an inertial observer.

\section{Formal framework}

Our framework is a fairly standard combination of model theory\footnote{See, \eg \citep{Ho93} or \citep{Ho97}.}, definability theory\footnote{See, \eg \citep{definability}.} and Tarskian algebraic logic\footnote{See \citep{HMT71}, \citep{HMTAN}, \citep{HMT85}, \citep{Monk2000}, and \citep{Andreka2022}.}, with some minor variations to the notation to suit our needs.

We use the following set of basic logical symbols for first-order predicate logic with equality $$\mathsf{Log}\de\{\,\exists, \land, \neg, (, ), =\,\}$$ and assume that there is a countable set $\Var$ of variables.

\begin{conv} We usually refer to arbitrary elements of $\Var$ by using indexes.  For the sake of simplicity, we
  fix a concrete ordering $\var_1,\var_2\dots,\var_i,\dots$ of the
  variables. When we would like to talk about $n$-many arbitrary
  variables from $\Var$, we use double indexes
  $i_1,\dots,i_n$. Sometimes the list of variables
  $\var_{i_1},\dots,\var_{i_n}$ is abbreviated to $\bar{v}$ and
  quantifiers $\forall\var_{i_1},\dots,\forall \var_{i_n}$ to $\forall
  \bar{v}$.  Sometimes, when the concrete value $i$ is not important, we use metavariables such as $x$, $y$, $z$ to
  denote $\vi$ for some $i$.
\end{conv}

A \emph{signature\footnote{A \textit{signature} is also called a \textit{vocabulary}.} of language $\LL$} is a pair $\langle \Pred_{\LL}, \ar_{\LL} \rangle$ of the set $\Pred_{\LL}$ of \emph{predicates}\footnote{Note that we allow $\Pred_\LL$ to be infinite.} (relation symbols) and the \emph{arity function} $\ar_{\LL}$ which assigns an arity\footnote{The \textit{arity} is the number of variables in the relation, it is also called the \textit{rank}, \textit{degree}, \textit{adicity} or \textit{valency} of the relation.} to elements of $\Pred_{\LL}$.
\emph{\textbf{Formulas} of language $\LL$} are built up recursively from alphabet $\Pred_{\LL}\cup \mathsf{Log} \cup \Var$ in the usual way and their set is denoted by $\Fm_{\LL}$.
A \emph{\textbf{model} $\mathfrak{M} = \langle M,\langle p^{\mathfrak{M}}: p\in \Pred_{\LL}\rangle \rangle$ of language $\LL$} consists of a non-empty \emph{underlying set} $M$, and for every predicate $p$ of $\LL$, a relation $p^{\mathfrak{M}} \subseteq M^n$ with the  arity $\ar_{\LL}(p) = n$.\footnote{The underlying set $M$ is also called the \textit{universe}, the \textit{carrier} or the \textit{domain} of model $\mathfrak{M}$. $M^n$ denotes the Cartesian power set of set $M$.}

By $\bar a ^i_b$ let us denote the sequence which is the same as $\bar a=(a_1,a_2,\dots,a_n,\dots)$ except at $i$ where it is $b$, \ie $\bar a^i_b=(a_1,\dots,a_{i-1},b,a_{i+1},\dots)$. When using a metavariable, say $x$ abbreviating $v_i$, we talk about the $x$-th component of $\bar a$ meaning the $i$-th component, and also use notation $\bar a^x_b$ instead of $\bar a^i_b$ in the same spirit.\footnote{See Figure~\ref{fig-xindep} below for an example on the usage of $\bar a ^i_b$.}

To recall the notion of \emph{semantics}, let $\mathfrak{M}$ be a model, let $M$ be the underlying set of $\mathfrak{M}$, let $\varphi$ be a formula and let $\bar a\in M^{\omega}$ be an infinite sequence of elements of $\mathfrak{M}$ then we inductively define that $\bar a$ \emph{\textbf{satisfies} $\varphi$ in} $\mathfrak{M}$, in symbols $\mathfrak{M}\models\varphi[\bar a],$ as:

\begin{enumerate}[(i)]
\item \label{semantics1} For predicate $p$, $\mathfrak{M}\models p(\var_{i_1}, \var_{i_2}, \dots, \var_{i_n})[\bar a]$ holds if $\big(a_{i_1}, a_{i_2}, \dots , a_{i_n}\big)\in p^\mathfrak{M},$ 
\item \label{semantics2} $\mathfrak{M}\models (\vi = \vj)[\bar a]$ holds if $a_i=a_j$ holds,
\item \label{semantics3} $\mathfrak{M}\models \neg\, \varphi[\bar a]$ holds if $\mathfrak{M}\models \varphi[\bar a]$ does not hold,
\item \label{semantics4} $\mathfrak{M}\models (\psi \land \theta)[\bar a]$ holds if both $\mathfrak{M}\models \psi[\bar a]$ and $\mathfrak{M}\models \theta[\bar a]$ hold,
\item \label{semantics5} $\mathfrak{M}\models \big(\exists \, \vj \, \psi\big)[\bar a]$ holds if there is an element $b \in M$, such that $\mathfrak{M}\models \psi\left[\bar a^j_b\right]$.
\end{enumerate}
$\mathfrak{M}\models\varphi[\bar{a}]$ can also be read as $\varphi[\bar{a}]$ \textit{being true in} $\mathfrak{M}$. That $\varphi$ is true in $\mathfrak{M}$ for all evaluations of variables is denoted by $\mathfrak{M}\models\varphi$. 

\begin{rem}\label{rem-abbr}
We use $\varphi \lor \psi$ as an abbreviation for $\neg\,(\neg\, \varphi \land \neg\, \psi)$, $\varphi \rightarrow \psi$ for $\neg\, \varphi \lor \psi$,  $\varphi \leftrightarrow \psi$ for $(\varphi \rightarrow \psi) \land (\psi \rightarrow \varphi)$, and $\forall \var_i \,\varphi$  for $\neg\, \exists\, \var_i \, \neg\, \varphi$. 
\end{rem}

Let $\MM$ be a model and $\varphi$ be a formula of its language. Then
the \emph{\textbf{meaning} of $\varphi$ in $\MM$} is defined as the set
of sequences from $\MM$ satisfying $\varphi$, \ie 
\begin{equation*}
  \mng{\varphi}{\MM}\defeq\{\bar{a}\in M^{\omega} : \MM\models \varphi[\bar{a}]\}.
\end{equation*}

  Let $x$ be a variable, let $\MM$ be a model, and let $\varphi$ and
  $\psi$ be a formulas of the language of $\MM$.  Then, by the
  definition of meaning, we have
\begin{equation*}
  \mng{\forall x\varphi}{\MM} \subseteq \mng{\varphi}{\MM} \subseteq
  \mng{\exists x \varphi}{\MM},
\end{equation*}
as illustrated in Figure~\ref{fig-meaning-cylindric}.

\begin{figure}
  \begin{center}
    \includegraphics{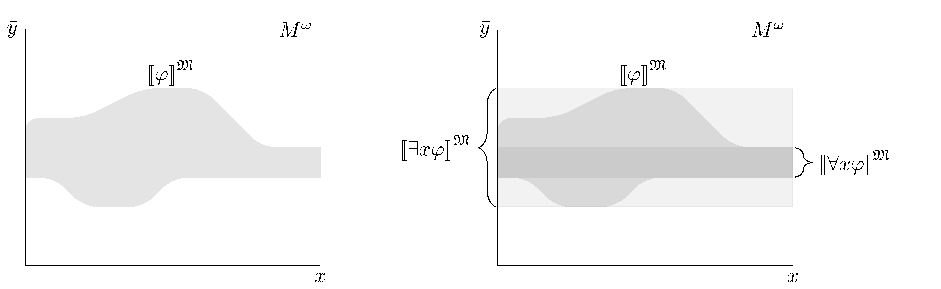}
  \caption{\label{fig-meaning-cylindric} \footnotesize{Let the medium grey $\mng{\varphi}{\MM}$ be the set of all values of $x$ and $\bar y$ in $M^{\omega}$ for which $\varphi$ is true. On the right, the meaning of ``exists'' (light grey rectangle, actually a cylinder with an infinite number of dimensions) and ``for all'' (dark grey rectangle) are added, illustrating that $\mng{\forall x\varphi}{\MM} \subseteq \mng{\varphi}{\MM} \subseteq \mng{\exists x \varphi}{\MM}$. Note that the axis $\bar y$ is represented as a vector because there are an infinite number of dimensions in $M^{\omega}$.}}
  \end{center}
\end{figure}

\begin{rem}\label{rem:mng-nice}
  There is a set theoretic operation corresponding to every logic
  operation behaving nicely with meanings:
  \begin{itemize}
    \item complement to negation $\mng{\lnot \varphi}{\MM}=
      M^\omega\setminus\mng{\varphi}{\MM}$, we will abbreviate this
      as $-\mng{\varphi}{\MM}$,
    \item intersection to conjunction $\mng{\varphi\land
      \psi}{\MM}=\mng{\varphi}{\MM}\cap\mng{\psi}{\MM}$,
  \item union to disjunction $\mng{\varphi\lor
    \psi}{\MM}=\mng{\varphi}{\MM}\cup\mng{\psi}{\MM}$,
  \item existential quantifiers to cylindrifications\footnote{For further discussion of the cylindrification $C_x$ see, e.g., \citep[p.~452, section 2]{Monk2000}.}
    \begin{equation*}
      \mng{\exists x \varphi}{\MM}=C_x\mng{\varphi}{\MM}=\left\{ \bar a \in M^\omega : \bar a^x_b\in\mng{\varphi}{\MM} \text{ for some }b\in M\right\},
    \end{equation*}
    see Figure~\ref{fig-meaning-cylindric}.
  \end{itemize}
\end{rem}

\section{Definitions and Theorems}
\label{sec:def}

Throughout this section, let $\MM$ be a model, $x$ and $y$ be variables, and let $\varphi$, $\psi$ and $\theta$ be formulas in the language of $\MM$. 

\begin{defn} We say that $\varphi$ is
\emph{\textbf{non-dependent of variable $x$} in model $\MM$} if{}f for
all sequences of elements $\bar a\in M^{\omega}$ and $b\in M$,
\begin{equation*}
  \MM\models\varphi[\bar a] \iff \MM\models\varphi[\bar a ^x_b].
\end{equation*}
\end{defn}

Let us note that we have the following equivalent\footnote{While we use single-line arrows $\leftrightarrow$ and $\rightarrow$ for equivalence and implication in the object language, we use double-line arrows $\iff$ and $\Longrightarrow$ in the meta-language.} formulations of variable non-dependence:
  \begin{equation*}\label{eq:indep}
    \varphi \text{ is non-dependent of $x$ in } \MM \iff \mng{\forall x
      \varphi}{\MM}=\mng{\varphi}{\MM} \iff
    \mng{\varphi}{\MM}=\mng{\exists x \varphi}{\MM},
  \end{equation*}
  and hence
   \begin{equation*}
         \varphi \text{ is non-dependent of $x$ in } \MM \iff \MM \models \exists x \varphi \leftrightarrow \forall x \varphi.
  \end{equation*}
This is a corollary of Proposition~\ref{prop:indep} below, and it can
be proven by choosing $\theta$ to be a tautology in that statement.

Let us note that if variable $x$ does not occur free in $\varphi$, then $\varphi$ is non-dependent of variable $x$ in every model. However, the converse does not hold: for example, the formula $x=x$ is non-dependent of variable $x$ in every model, but $x$ does occur free in it.

\begin{defn}\label{def:indep-prov} We say that $\varphi$ is
\emph{\textbf{non-dependent of variable $x$ in model $\MM$ provided
    $\theta$}} if{}f, for all sequences of elements $\bar a\in
M^{\omega}$ and $b\in M$,
\begin{equation}\label{eq:indep-prov}
  \MM\models\theta [\bar a]\ \text{ and }\  \MM\models\theta [\bar a^x_b]\ \implies\ (\ \MM\models\varphi[\bar a] \iff \MM\models\varphi[\bar a ^x_b]\ ), 
\end{equation}
see Figure~\ref{fig-xindep}.
\end{defn}

\begin{rem}\label{rem:2} It is straightforward to check the following observations from the definitions:
\begin{itemize}
  \item $\theta$ is non-dependent of $x$ in $\MM$ provided $\theta$,
  \item $\exists x \varphi$ is always non-dependent of $x$ in $\MM$,
  \item if $\varphi$ is non-dependent of $x$ in $\MM$ provided $\theta$, then so is $\exists y \varphi$,
  \item Boolean-closedness: if $\varphi$ and $\psi$ are non-dependent of $x$ in $\MM$ provided $\theta$, then so are $\neg \varphi$ and $\varphi \land \psi$,
  \item monotonicity: if $\varphi$ is non-dependent of $x$ in $\MM$ provided $\theta$ and $\hat{\theta} $ implies $\theta$ in $\MM$, then $\varphi$ is non-dependent of $x$ in $\MM$ provided $\hat{\theta}$.
  \end{itemize}
  \end{rem}
 \begin{rem}\label{rem:boolean} 
  By Boolean-closedness and monotonicity, we have the following: In any model, if $\varphi_1$ is non-dependent of $x$ provided  $\theta_1$ and $\varphi_2$ is non-dependent of $x$ provided $\theta_2$, then $\varphi_1*\varphi_2$ is non-dependent of $x$  provided $\theta_1\land\theta_2$ for any binary Boolean-definable logical connective $*$.
\end{rem}

\begin{figure}
  \begin{center}
    \includegraphics{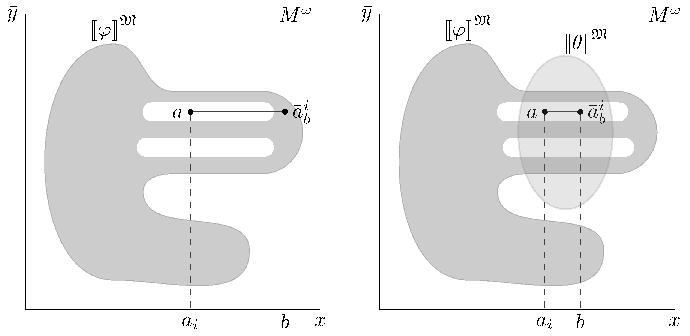}
    \caption{\label{fig-xindep} \footnotesize{On the left hand, we see a formula $\varphi$ which is not non-dependent of $x$: changing the $x$-value from $a_i$ into $b$ changes the truth value of $\varphi$. However, on the right we see that adding an extra condition $\theta$ does make $\varphi$ non-dependent of $x$ provided $\theta$: changing the $x$-value does never change the truth value of $\varphi$ as long as the evaluations of variables remain inside the area defined by $\mng{\theta}{\MM}$.}}
  \end{center}
\end{figure}

For arbitrary formulas $\phi$ and $\varphi$, we are using bounded quantifiers as follows:\footnote{We here use a notation for bounded quantifier where the bounds are other formulas viewed as parametricaly defined subsets of the model where we interpret them.  An advantage of this notation is that  it makes the ideas behind some formulas easier to grasp.  Similar notation can be found in, e.g., \citep{BigBook} and \citep{logst}; we already used this similar notation in the translation examples in the introduction, for example, $\bforall{k}{IOb}\varphi$ instead of $\forall k (IOb(k) \to \varphi)$.}
\begin{equation}\label{eq-bquant}
\bforall {u}{\phi}\varphi \defiff \forall u (\phi \to \varphi)\  \text{ and }\ 
\bexists {u}{\phi}\varphi \defiff \exists u (\phi \AND \varphi). 
\end{equation}

\begin{prop}\label{prop:indep}
  The following statements are equivalent:
  \begin{enumerate}[(i)]
  \item\label{i-pindep} $\varphi$ is non-dependent of $x$ in $\MM$ provided $\theta$,
  \item\label{i-epindep}
    $\mng{\theta\land\bexists{x}{\theta}\varphi)}{\MM}=\mng{\theta\land\varphi}{\MM}$,  
  \item\label{i-upindep0} $\mng{\theta\land\bforall{x}{\theta}
    \varphi)}{\MM}=\mng{\theta\land\varphi}{\MM}$,  
  \item\label{i-upindep} $\mng{\theta \to \bforall{x}{\theta}
    \varphi)}{\MM}= \mng{\theta\to \varphi}{\MM}$, and
  \item\label{i-epindep0} $\mng{\theta\to\bexists{x}{\theta}
      \varphi)}{\MM}=\mng{\theta\to\varphi}{\MM}$.
  \end{enumerate}
\end{prop}
\noindent
Let us note here that
\begin{equation}\label{eq-ingeneral}
  \mng{\theta\land \exists{x}(\theta\land \varphi)}{\MM} \supseteq
  \mng{\theta\land\varphi}{\MM} \subseteq \mng{\varphi}{\MM} \subseteq
  \mng{\theta\to\varphi}{\MM} \supseteq
  \mng{\theta\to\forall{x}(\theta\to\varphi)}{\MM}
\end{equation}
holds in general.

\begin{figure}[h!tb]
  \begin{center}
    \begin{tikzpicture}[scale=1.05]
 
      \tikzstyle{sago}=[->,double, thick]
      \tikzstyle{nodo}=[inner sep=6]
      
      \node[nodo] (i) at (4,0) {(i)};
      \node[nodo] (ii) at (1,0) {(ii)};
      \node[nodo] (iii) at (7,0) {(iii)};
      \node[nodo] (iv) at (4,2) {(iv)};
      \node[nodo] (v) at (4,-2) {(v)};

      \draw[sago,<->]  (i) to (ii) ;
      \draw[sago] (iii) to (i) ;
      \draw[sago] (ii) -- (iv) ;
      \draw[sago] (iv) to (iii); 
      \draw[sago] (iii) -- (v) ;
      \draw[sago] (v) to (ii);
    \end{tikzpicture}
  \end{center}
  \caption{This figure illustrates the order of proving the equivalences between the items of Proposition~\ref{prop:indep}. We need to prove ``(iii)$\implies$(i)'' directly because in the proof of ``(iii)$\implies$(v)'' we use the equivalence of (iii) and (i).}
\end{figure}
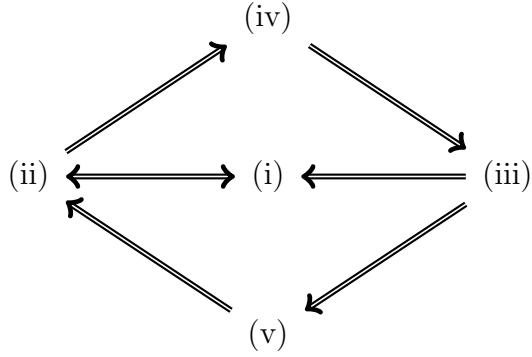

\begin{proof}
  Proof of ``\eqref{i-pindep}$\implies$\eqref{i-epindep}'': By
  \eqref{eq-ingeneral} and \eqref{eq-bquant}, it is enough to show that $\mng{\theta\land
    \exists{x}(\theta\land \varphi)}{\MM} \subseteq
  \mng{\theta\land\varphi}{\MM}$. To do so, let $\bar a\in
  \mng{\theta\land \exists{x}(\theta\land \varphi)}{\MM}$. Then $\bar a\in
  \mng{\theta}{\MM}$ and $\bar a\in \mng{\exists{x}(\theta\land
    \varphi)}{\MM}$, the latter means that there is $b\in M$ such that
  $\bar a^x_b\in \mng{\theta}{\MM}\cap \mng{\varphi}{\MM}$.  From
  this, by Definition~\ref{def:indep-prov}, follows that $\bar a\in
  \mng{\varphi}{\MM}$. Consequently, $\bar a\in
  \mng{\theta\land\varphi}{\MM}$, and this is what we wanted to show.

  Proof of ``\eqref{i-epindep}$\implies$\eqref{i-pindep}'': We are going to prove the contraposition of the statement. So assume that  \eqref{i-pindep} does not hold, \ie there is $\bar a \in M^\omega$  and $b\in M$ such that $\MM\models \theta[\bar a]$, $\MM\models \theta[\bar a ^x_b]$, $\MM\models \varphi[\bar a^x_b]$, but $\MM\not\models \varphi[\bar a]$. Then $\bar a \not\in \mng{\theta\land\varphi}{\MM}$, but $\bar a\in\mng{\exists x (\theta\land\varphi)}{\MM}$. Hence, by $\MM\models \theta[\bar a]$, we get $\bar a\in\mng{\theta\land\exists x(\theta\land\varphi)}{\MM}$. Thus $\mng{\theta\land\exists x(\theta\land\varphi)}{\MM}\neq  \mng{\theta\land\varphi}{\MM}$. Proving that if \eqref{i-pindep} does not hold, then \eqref{i-epindep} also does not hold. Consequently, \eqref{i-epindep} implies \eqref{i-pindep} as stated.
  
  Proof of ``\eqref{i-epindep}$\implies$\eqref{i-upindep}'' after using equation \eqref{eq-bquant} to unpack the bounded quantifier:
  \begin{align*}
    &\mng{\theta \to \forall{x}(\theta \to \varphi)}{\MM} \\
    &=\mng{\lnot\theta \lor \neg \exists{x}\lnot(\lnot\theta \lor \varphi)}{\MM} && \text{by the definitions of $\forall$ and $\to$.}\\
    &=\mng{\lnot\big(\theta\land \exists{x}(\theta \land \lnot\varphi)\big)}{\MM} && \text{by De Morgan twice and double negation.\footnotemark}\\
    &=-\mng{\theta \land \exists{x}(\theta\land \lnot\varphi)}{\MM} && \text{by Remark~\ref{rem:mng-nice}.}\\
    &=-\mng{\theta \land\lnot\varphi}{\MM} &&\text{by item \eqref{i-epindep} on $\neg\varphi$, \eqref{i-epindep}$\Leftrightarrow$\eqref{i-pindep} and Remark~\ref{rem:2}.}\\
  &=\mng{\neg\theta \lor \neg\neg\varphi}{\MM} &&\text{by Remark~\ref{rem:mng-nice} and De Morgan.}\\
      &=\mng{\theta \to\varphi}{\MM} &&\text{by double negation and definition of $\to$.}
  \end{align*}\footnotetext{See, e.g., \citep[p.~34]{hinman} for tautologies in propositional logic that we use, including De Morgan's laws, double negation, excluded middle, associativity, distributivity, exportation, and idempotency.}

  Proof of ``\eqref{i-upindep}$\implies$\eqref{i-upindep0}'' after unpacking the bounded quantifier:
  \begin{align*}
    & \mng{\theta\land\forall{x}(\theta\to \varphi)}{\MM} \\
    & =\mng{\theta\land(\theta\to\forall{x}(\theta\to   \varphi))}{\MM} &&\text{by identity } A\land B \equiv A \land (A \to B).\\
    & =\mng{\theta\land(\theta\to \varphi)}{\MM} &&\text{by item \eqref{i-upindep}}.\\
    & =\mng{\theta\land \varphi}{\MM} &&\text{by identity }A \land (A \to B)\equiv A\land B.
  \end{align*}

  Proof of ``\eqref{i-upindep0}$\implies$\eqref{i-pindep}'': we prove the contraposition of the statement. To do so, assume that  \eqref{i-pindep} does not hold, \ie there is $\bar a \in M^\omega$  and $b\in M$ such that $\MM\models \theta[\bar a]$, $\MM\models \theta[\bar a ^x_b]$, $\MM\not\models \varphi[\bar a^x_b]$, but $\MM\models \varphi[\bar a]$.
    Then $\bar a \in \mng{\theta\land\varphi}{\MM}$, but $\bar a\not\in\mng{\forall x (\theta\land\varphi)}{\MM}$, and hence $\bar a\not\in\mng{\theta\land\forall x(\theta\land\varphi)}{\MM}$. Thus $\mng{\theta\land\forall x(\theta\land\varphi)}{\MM}\neq  \mng{\theta\land\varphi}{\MM}$, and this is what we wanted to show.
  
Proof of ``\eqref{i-upindep0}$\implies$\eqref{i-epindep0}'' after unpacking the bounded quantifier:
  \begin{align*}
    & \mng{\theta\to\exists{x}(\theta\land \varphi)}{\MM}\\
    &=\mng{\lnot\theta \lor \exists{x}(\theta \land \varphi)}{\MM} && \text{by the definition of $\to$.}\\
    &=\mng{\lnot\big(\theta\land \neg\exists{x}(\theta \land \varphi)\big)}{\MM} && \text{by De Morgan and double negation.}\\
    &=-\mng{\theta \land \forall{x}\lnot(\theta\land \varphi)}{\MM} && \text{by Remark~\ref{rem:mng-nice} and quantifier negation law.\footnotemark}\\
    &=-\mng{\theta \land \forall{x}(\theta\to \neg \varphi)}{\MM} && \text{by De Morgan and the definition of $\to$.}\\
    & =-\mng{\theta \land \neg \varphi}{\MM} &&\text{by item \eqref{i-upindep0} on $\neg\varphi$, \eqref{i-upindep0}$\Leftrightarrow$\eqref{i-pindep} and Remark~\ref{rem:2}.}\\
    & =\mng{\neg\theta \lor \neg\neg\varphi}{\MM} &&\text{by Remark~\ref{rem:mng-nice} and De Morgan.}\\
    & =\mng{\theta\to\varphi}{\MM} &&\text{by double negation and definition of $\to$.} 
\end{align*}\footnotetext{See, e.g., \citep[p.~99]{hinman} for equivalences in first-order logic that we use, including negation and distributivity of quantifiers.}

Proof of ``\eqref{i-epindep0}$\implies$\eqref{i-epindep}'' after unpacking the bounded quantifier:
  \begin{align*}
    & \mng{\theta\land\exists{x}(\theta\land \varphi)}{\MM}\\
    & =\mng{\theta\land(\theta\to\exists{x}(\theta\land \varphi))}{\MM} &&\text{by identity } A\land B \equiv A \land (A \to B).\\
    & =\mng{\theta\land(\theta\to \varphi)}{\MM} &&\text{by item \eqref{i-epindep0}}.\\
    & =\mng{\theta\land \varphi}{\MM} &&\text{by identity }A \land (A \to B)\equiv A\land B.
  \end{align*}
\end{proof}

\begin{prop}\label{prop:2}
  If $\varphi$ is non-dependent of $x$ in $\MM$ provided
  $\theta$, then $\mng{\bexists{x}{\theta}\neg\varphi)}{\MM}$ is the
  complement of $\mng{\bexists{x}{\theta}\varphi)}{\MM}$ relative to $
  \mng{\exists x \theta}{\MM}$, \ie
    \begin{equation*}
      \mng{\bexists {x}{\theta} \neg\varphi)}{\MM} = \mng{\exists x
        \theta}{\MM} - \mng{\bexists x{\theta}\varphi)}{\MM}
    \end{equation*}
    in other words
    \begin{equation*}
      \mng{\bexists {x}{\theta} \neg\varphi)}{\MM} = 
    \mng{\exists x \theta \land \neg \bexists{x}{\theta} 
      \varphi)}{\MM}.
  \end{equation*}
\end{prop}

\begin{proof}
  If $\varphi$ is non-dependent of $x$ in $\MM$ provided $\theta$, then
  \begin{equation*}
    \mng{\bexists{x}{\theta} \varphi)}{\MM} \cap \mng{\bexists{
        x}{\theta} \neg \varphi)}{\MM} = \varnothing.
  \end{equation*}
  This is so because, if there was an $\bar{a}\in\mng{\exists x(\theta
    \land \varphi)}{\MM} \cap \mng{\exists x(\theta \land \neg
    \varphi)}{\MM}$, then there would also be $b$ and $c$ such that
  $\bar{a}^x_b$ and $\bar{a}^x_c$ satisfy $\theta$, and $\MM\models
  \varphi[\bar{a}^x_b]$, but $\MM\models
  \neg\varphi[\bar{a}^x_c]$. This would contradict the assumption that
  $\varphi$ is non-dependent of $x$ in $\MM$ provided $\theta$. Hence
  the intersection of $\mng{\exists x(\theta \land \varphi)}{\MM}$ and
  $\mng{\exists x(\theta \land \neg \varphi)}{\MM}$ has to be
  empty.

  To complete the proof, now we are going to show their
    union is $\mng{\exists x \theta}{\MM}$.
\begin{align*}
  &\mng{\exists x(\theta \land \varphi)}{\MM} \cup \mng{\exists x(\theta \land \neg \varphi)}{\MM} \\
  & = \mng{\exists x(\theta \land \varphi) \lor \exists x(\theta \land \neg \varphi)}{\MM} &&\text{by Remark~\ref{rem:mng-nice}.}\\
  & = \mng{\exists x\big((\theta \land \varphi) \lor (\theta \land \neg \varphi)\big)}{\MM} && \text{by the distributivity of $\exists$ over $\lor$.}\\
& = \mng{\exists x \big(\theta \land (\varphi \lor \neg \varphi)\big)}{\MM} &&\text{by the distributivity of $\land$ over $\lor$.}\\
& = \mng{\exists x \theta}{\MM} &&\text{by excluded middle.}
\end{align*}
\end{proof}

\begin{rem}\label{rem-hinman} 
  Let us recall the following facts from the literature:\footnote{See, \eg \cite[p.99]{hinman}.} If variable $x$ does not occur free in $\phi$, then we have the following logical equivalences:
  \begin{enumerate}[(i)]
    \centering
   \item \label{hm-e} $\exists x (\phi \land \psi )\equiv \phi \land \exists x\, \psi$,
   \item  \label{hm-a} $\forall x (\phi \lor \psi)  \equiv \phi \lor \forall x\, \psi$,
   \item \label{hm-et} $\exists x (\phi\to \psi)\equiv \phi \to \exists x \psi$,
   \item  \label{hm-at} $\forall x (\phi\to \psi)\equiv \phi \to \forall x \psi$,
   \item \label{hm-ta} $\forall x (\psi \to \phi) \equiv \exists x\psi \to \phi$.
  \end{enumerate}
\end{rem}

\begin{prop}\label{prop:3} 
If $\varphi$ is non-dependent of $x$ in $\MM$ provided $\theta$, then 
\begin{equation*}
  \mng{\bforall {x}{\theta} \neg \varphi)}{\MM} = \mng{\exists x \theta \rightarrow \neg \bforall{x}{\theta} \varphi) }{\MM}.
\end{equation*}

\end{prop}
\begin{proof}
If $\varphi$ is non-dependent of $x$ in $\MM$ provided $\theta$, then after unpacking the bounded quantifier 
\begin{align*}
& \mng{\forall x(\theta \rightarrow \neg \varphi)}{\MM} \\
  & =\mng{\neg \exists x\neg(\neg\theta \lor \lnot\varphi)}{\MM} &&\text{by the definitions of $\forall$ and $\to$.}\\
& =\mng{\neg \exists x(\theta \land \varphi)}{\MM} &&
  \text{by De Morgan, and double negation.}\\
  & =-\mng{\exists x(\theta \land \varphi)}{\MM} &&\text{by Remark~\ref{rem:mng-nice}.}\\
& =-\mng{\exists x\big(\theta\land\forall{x}(\theta\to
    \varphi)\big)}{\MM} &&  \text{by  \eqref{i-upindep0} of Proposition~\ref{prop:indep}.}\\
  & =-\mng{\exists x\big(\theta\land\lnot\exists{x}\lnot(\lnot\theta\lor \varphi)\big)}{\MM} &&  \text{by the definitions of $\forall$ and $\to$.}\\
    & =-\mng{\exists x\big(\theta\land\lnot\exists{x}(\theta\land \lnot\varphi)\big)}{\MM} &&  \text{by De Morgan and double negation.}\\
  & =-\mng{\exists x\theta\land\lnot\exists{x}(\theta\land \lnot
    \varphi)}{\MM} &&  \text{by \eqref{hm-e} of Remark~\ref{rem-hinman}.}\\
  & =\mng{\lnot\big(\exists x\theta\land\lnot\exists x(\theta \land \neg \varphi)}{\MM} &&\text{by Remark~\ref{rem:mng-nice}.}\\
  & =\mng{\lnot\exists x\theta\lor\exists x(\theta \land \neg \varphi)}{\MM}  &&\text{by De Morgan and double negation.}\\
  & =\mng{\lnot\exists x\theta\lor\lnot\forall x\lnot(\theta \land \neg \varphi)}{\MM}  &&\text{by double negation and definition of $\forall$.}\\
  & =\mng{\lnot\exists x\theta\lor\lnot\forall x (\lnot\theta \lor \varphi)}{\MM}  &&\text{by De Morgan and double negation.}\\
& =\mng{\exists x\theta \rightarrow \neg \forall x (\theta \rightarrow \varphi)}{\MM} &&\text{by definition of $\to$.}
\end{align*}
\end{proof}

From Propositions~\ref{prop:2} and \ref{prop:3}, we get the following:
\begin{cor}
  If $\varphi$ is non-dependent of $x$ in $\MM$ provided $\theta$ and $\MM\models\exists x \theta$, \ie $\mng{\exists x\theta}{\MM}=M^\omega$, then
  \begin{align*}
    \mng{\neg\bexists{x}{\theta}\varphi}{\MM}&=\mng{\bexists{x}{\theta}\neg\varphi)}{\MM}, \text{ and}\\
    \mng{\neg\bforall{x}{\theta}\varphi}{\MM}&=\mng{\bforall{x}{\theta}\neg\varphi)}{\MM}.
  \end{align*}
\end{cor}

  \begin{prop}\label{prop-h}
    If $\varphi$ is non-dependent of $x$ in $\MM$ provided $\theta$ and $\MM\models\exists x \theta$, then
    \begin{equation*}
       \mng{\bexists{x}{\theta}\varphi}{\MM}=\mng{\bforall{x}{\theta}\varphi)}{\MM}.
    \end{equation*}
  \end{prop}
  \begin{proof}
  Let $\bar a \in \mng{\bexists{x}{\theta}\varphi}{\MM}$. Then,
     for some $b\in M$, $\bar a^x_b \in \mng{\theta}{\MM}\cap
     \mng{\varphi}{\MM}$ by definitions. Let $c\in M$ be arbitrary.
     Since $\varphi$ is non-dependent of $x$ in $\MM$ provided
     $\theta$, we have that if $\bar a^x_c\in \mng{\theta}{\MM}$, then
     $\bar a^x_c\in \mng{\varphi}{\MM}$ as $\bar a^x_b \in
     \mng{\theta}{\MM}$ and $\bar a^x_b \in \mng{\varphi}{\MM}$.
     Then, since either $\bar a^x_c \in \mng{\neg\theta}{\MM}$ or
     $\bar a^x_c \in \mng{\theta}{\MM}$, we have $\bar a^x_c \in
     \mng{\neg\theta}{\MM} \cup
     \mng{\varphi}{\MM}=\mng{\theta\to\varphi}{\MM}$. Because $c$ was
     arbitrary, this means that $\bar a\in \mng{\forall x
       (\theta\to\varphi)}{\MM} =
     \mng{\bforall{x}{\theta}\varphi)}{\MM}$. This proves inclusion
     $\mng{\bexists{x}{\theta}\varphi}{\MM} \subseteq
     \mng{\bforall{x}{\theta}\varphi)}{\MM}$.

     To prove the other inclusion, let $\bar a \in
     \mng{\bforall{x}{\theta}\varphi)}{\MM}$, \ie for all $c\in M$, if
     $\bar a^x_c \in \mng{\theta}{\MM}$ holds, then so does $\bar
     a^x_c \in \mng{\varphi}{\MM}$. By assumption $\MM\models\exists x
     \theta$, there is some $b\in M$ such that $\bar a^x_b\in
     \mng{\theta}{\MM}$. By the above, for this $b$, we also have
     $\bar a^x_b \in \mng{\varphi}{\MM}$. In other words, $\bar a\in
     \mng{\exists x (\theta \land
       \varphi)}{\MM}=\mng{\bexists{x}{\theta}\varphi}{\MM}$, which
     proves the other inclusion
     $\mng{\bexists{x}{\theta}\varphi}{\MM}\supseteq\mng{\bforall{x}{\theta}\varphi)}{\MM}$.
  \end{proof} 
  
We note that condition $\MM\models\exists x \theta$ is needed for inclusion $\mng{\bexists{x}{\theta}\varphi}{\MM}\supseteq\mng{\bforall{x}{\theta}\varphi)}{\MM}$ and the non-dependence condition is needed for inclusion $\mng{\bexists{x}{\theta}\varphi}{\MM}\subseteq\mng{\bforall{x}{\theta}\varphi)}{\MM}$.

\begin{prop}\label{prop-B} Bounded universal quantifiers distribute over conjunction, \ie 
 \begin{align*}
    \bforall{x}{\phi}(\varphi \land \psi) &\equiv \bforall{x}{\phi}  \varphi \land \bforall{x}{\phi} \psi, \ \text{ and hence}\\
    \mng{\bforall{x}{\phi}(\varphi \land \psi)}{\MM} &= \mng{\bforall{x}{\phi}  \varphi \land \bforall{x}{\phi} \psi}{\MM}.
  \end{align*}
\end{prop}
\begin{proof} After unpacking the bounded quantifiers, the statement can be proved\footnote{While the proof of this is straightforward, we include it here due to our peculiar use of bounded quantifiers.} as: 
\begin{align*}
& \forall x \big(\phi \rightarrow (\varphi \land \psi)\big) \\
& \equiv \forall x \big(\neg \phi \lor (\varphi \land \psi)\big) &&\text{by the definition of implication.}\\
& \equiv \forall x \big((\neg \phi \lor \varphi)\land(\neg \phi \lor \psi)\big) &&\text{by the distributivity of $\lor$ over $\land$.}\\
& \equiv \forall x (\neg \phi \lor \varphi)\land\forall x(\neg \phi \lor \psi) &&\text{by the distributivity of $\forall$ over $\land$.} \\
& \equiv \forall x (\phi \rightarrow \varphi) \land \forall x (\phi \rightarrow \psi) &&\text{by the definition of implication.}
\end{align*}
\end{proof}

In general, quantifiers do not distribute over logic operators. For example, $\forall x\big(\phi(x) \lor \psi(x)\big)$ has a different meaning than $\forall x \big(\phi(x)\big) \lor \forall x \big(\psi(x)\big)$, which is clear when we consider that ``\textit{all numbers are odd or even}'' is very different from ``\textit{all numbers are odd or all numbers are even}''. However, under certain conditions, using Proposition~\ref{prop-C} below, it is possible to bring out quantifiers which are nested within operators\ ---\ in the case of the above example, if $\phi$ and $\psi$ are non-dependent\footnote{This is obviously not the case for ``$x$ is odd'' and ``$x$ is even''.} of $x$, by assigning the function $f\big(\phi(x), \psi(x)\big)$ in Proposition~\ref{prop-C} to the logical \textit{or} $\lor$.

Now, in Proposition~\ref{prop-C}, we are going to prove that bounded universal quantifier
  $\bforall{x}{\theta}$ can be brought out from arbitrary boolean
  combination of formulas if they are non-dependent of variable $x$
  provided $\theta$.

\begin{prop}\label{prop-C}
Let $f$ be any boolean expression, then
\begin{equation*}
  \mng{\exists x \theta \rightarrow f\big(\bforall {x}{\theta} \varphi_1, \dots, \bforall {x}{\theta} \varphi_n\big)}{\MM} = \mng{\bforall{x}{\theta} f( \varphi_1, \dots, \varphi_n)}{\MM}
\end{equation*}
if all $\varphi_i$ are non-dependent of $x$ in $\MM$ provided
$\theta$.
\end{prop}

\begin{proof}
  We prove the statement by induction on the complexity of $f(\varphi_1, \dots, \varphi_n)$. Let us first show that the statement holds for each $\varphi_i$ in $f(\varphi_1, \dots, \varphi_n)$. Let $\varphi$ be any of those $\varphi_i$s. By Remark~\ref{rem:mng-nice}, we have $\mng{\theta \land \neg \varphi}{\MM}=\mng{\theta}{\MM}\cap\mng{\neg \varphi}{\MM}$. So $\mng{\theta \land \neg \varphi}{\MM} \subseteq\mng{\theta}{\MM}$. From this, by Remark~\ref{rem:mng-nice},  we get $\mng{\neg\exists x(\theta \land \neg \varphi)}{\MM}\supseteq\mng{\neg\exists x \theta}{\MM}$. Which is the same as $\mng{\forall x(\theta \to \varphi)}{\MM}\supseteq\mng{\neg\exists x \theta}{\MM}$  by the definitions of $\to$ and $\forall$, De Morgan and double negation. Hence $\mng{\forall x(\theta \to \varphi)}{\MM}=\mng{\neg\exists x \theta}{\MM}\cup \mng{\forall x(\theta \to \varphi)}{\MM}$. Which is  $\mng{\forall x(\theta \to \varphi)}{\MM}=\mng{\neg\exists x \theta\lor \forall x(\theta \to \varphi)}{\MM}$ by Remark~\ref{rem:mng-nice}. From this, by the definition of $\to$ and \eqref{eq-bquant}, we get the desired identity $\mng{\bforall{x}{\theta}\varphi}{\MM}=\mng{\exists x \theta\to \bforall{x}{\theta}\varphi}{\MM}$.

    Since, by Remark~\ref{rem-abbr}, $f$ is equivalent to an expression in which only negation and conjunction is used, it is enough to show the induction steps for these two connectives.

  Let us first assume that $f$ is of the form $f = g\land h$ such that we already know the statement for $g$ and $h$, \ie the followings hold
  \begin{align}\label{ind-g}
    \mng{\bforall{x}{\theta} g( \varphi_1, \dots, \varphi_n)}{\MM} &= \mng{\exists x \theta \rightarrow g\big(\bforall {x}{\theta} \varphi_1, \dots, \bforall {x}{\theta} \varphi_n\big)}{\MM}, \text{ and}     \\
\label{ind-h}
    \mng{\bforall{x}{\theta} h( \varphi_1, \dots, \varphi_n)}{\MM} &= \mng{\exists x \theta \rightarrow h\big(\bforall {x}{\theta} \varphi_1, \dots, \bforall{x}{\theta} \varphi_n\big)}{\MM}.     
  \end{align}

   \begin{align*}
     & \mng{\bforall{x}{\theta} \big( g(\varphi_1, \dots, \varphi_n) \land h(\varphi_1, \dots, \varphi_n) \big)}{\MM} \\
     & = \mng{\bforall{x}{\theta} g(\varphi_1, \dots, \varphi_n)}{\MM} \cap \mng{\bforall{x}{\theta} h(\varphi_1, \dots, \varphi_n)}{\MM} &&\text{ by Prop.~\ref{prop-B} and Rem.~\ref{rem:mng-nice}.} \\
  & = \mng{\exists x \theta \rightarrow g\big(\bforall {x}{\theta} \varphi_1, \dots, \bforall{x}{\theta} \varphi_n\big)}{\MM} && \\
  &\qquad\cap \mng{\exists x \theta \rightarrow h\big(\bforall {x}{\theta} \varphi_1, \dots, \bforall{x}{\theta} \varphi_n\big)}{\MM}  &&\text{by ind.\ hypotheses: \eqref{ind-g} and \eqref{ind-h}.}\\
     & = \mng{\big(\neg \exists x \theta \lor g(\bforall {x}{\theta} \varphi_1, \dots)\big)\land \big( \neg\exists x \theta \lor  h(\dots)\big)}{\MM} &&\text{by the definition of $\to$ and Rem.\ref{rem:mng-nice}.}\\
  & = \mng{\neg \exists x \theta \lor \big(g(\bforall {x}{\theta} \varphi_1,\dots)\big)\land h(\bforall {x}{\theta} \varphi_1,\dots)\big)}{\MM}  &&\text{by  the distributivity of $\lor$ over $\land$.}\\
  & = \mng{\exists x \theta \rightarrow (g \land h)\big(\bforall {x}{\theta} \varphi_1, \dots, \bforall {x}{\theta} \varphi_n\big)}{\MM} &&\text{by the definition of $\to$.}
\end{align*}

Let us now assume that $f$ is of the form $f = \neg g$ such that we already know the statement for $g$.

\begin{align*}
  &\mng{\bforall {x}{\theta} \neg g(\varphi_1, \dots, \varphi_n) }{\MM}\\
  &= \mng{\exists x \theta \rightarrow \neg \bforall{x}{\theta} g(\varphi_1, \dots, \varphi_n) }{\MM}  &&\text{by Prop.~\ref{prop:3} on $g(\varphi_1...)$ and Remark~\ref{rem:2}.} \\
  &= - \mng{\exists x \theta}{\MM} \cup - \mng{\bforall{x}{\theta} g(\varphi_1, \dots, \varphi_n) }{\MM} &&\text{by Remark~\ref{rem:mng-nice} and the definition of $\to$.}\\
  &= - \mng{\exists x \theta}{\MM} \cup - \mng{\exists x  \theta \rightarrow g\big(\bforall{x}{\theta} \varphi_1, \dots\big) }{\MM} &&\text{by induction hypothesis: \eqref{ind-g}.}\\
   &= \mng{\neg \exists x \theta\lor \neg\big(\neg\exists x  \theta \lor g\big(\bforall {x}{\theta} \varphi_1, \dots\big)\big) }{\MM} &&\text{by the definition of $\to$ and Remark~\ref{rem:mng-nice}.}\\
   &= \mng{\neg \exists x \theta\lor \big(\exists x  \theta \land\neg g\big(\bforall {x}{\theta} \varphi_1, \dots\big)\big) }{\MM} &&\text{by De Morgan and double negation.}\\
  &= \mng{(\neg \exists x \theta\lor\exists x  \theta) \land\big(\neg \exists x \theta\lor \neg g(\bforall {x}{\theta} \varphi_1,\dots)\big) }{\MM} &&\text{by the distributivity of $\lor$ over $\land$.}\\
       &= \mng{\neg \exists x \theta\lor \neg g\big(\bforall {x}{\theta} \varphi_1, \dots, \bforall {x}{\theta} \varphi_n\big) }{\MM} &&\text{by excluded middle.}\\
     &= \mng{ \exists x \theta\to \neg g\big(\bforall {x}{\theta} \varphi_1, \dots, \bforall {x}{\theta} \varphi_n\big) }{\MM} &&\text{by the definition of $\to$.}
\end{align*}
\end{proof}

  \begin{lemma}\label{prop-exists}  Assume that $\varphi$ is non-dependent of $x$ in $\MM$ provided $\theta$ and none of variables in $\bar z$  occur free in $\theta$. Then
    \begin{equation}
      \mng{\forall x \exists \bar z (\theta \to \psi)}{\MM}=\mng{\exists x \theta \to \exists \bar z \bforall{x}{\theta} \psi) }{\MM},
    \end{equation}
    and hence,
      \begin{equation}
      \mng{\forall x \exists \bar z (\theta \to \psi)}{\MM}=\mng{\exists \bar z \bforall{x}{\theta}\psi) }{\MM}\      \text{ if }\  \MM\models \exists x \theta.
      \end{equation}
  \end{lemma}
  \begin{proof}
  \begin{align*}
    &\mng{\forall x \exists \bar z (\theta \to \psi)}{\MM} \\
    &=\mng{\forall x \exists \bar z \big(\theta \to  \forall x (\theta \to \psi)\big) }{\MM} && \text{by item \eqref{i-upindep} of Prop.\ref{prop:indep}.}\\
    &=\mng{\forall x \big(\theta \to  \exists \bar z\forall x (\theta \to \psi)\big) }{\MM} && \text{by \eqref{hm-et} Remark~\ref{rem-hinman}.}\\
    &=\mng{\exists x \theta \to \exists \bar z \forall x (\theta \to \psi) }{\MM} && \text{by \eqref{hm-ta} Remark~\ref{rem-hinman}.}
  \end{align*}
  \end{proof}

Now we are going to show that bounded quantifier $\bforall{x}{\theta}$ can be brought out from any formula built up from subformulas all of which are non-dependent of $x$ provided $\theta$. We only need to prove this for formulas in prenex normal form, since every formula of first-order logic can be written as such, see Theorem 2.2.34 in \cite[p.111]{hinman}.
  
  \begin{thm}\label{prop-new}
    Let $f$ be any boolean expression, let $\theta$ be formula such that no variables of $z_1,\dots,z_m$ occur free in $\theta$,  and let $Q_1,\dots, Q_m$ be an arbitrary series of universal and existential quantifiers. Then, if all $\varphi_i$ are non-dependent of $x$ in $\MM$ provided $\theta$,
\begin{multline*}\mng{\exists x \theta \rightarrow  Q_mz_m\dots Q_1 z_1 f\big(\bforall {x}{\theta} \varphi_1, \dots, \bforall {x}{\theta} \varphi_n\big)}{\MM} \\= \mng{\bforall {x}{\theta} Q_mz_m\dots Q_1 z_1 f( \varphi_1, \dots, \varphi_n)}{\MM},
    \end{multline*}
    and hence, if $\MM\models \exists x \theta$, then
  \begin{equation*}\mng{Q_mz_m\dots Q_1 z_1 f\big(\bforall {x}{\theta} \varphi_1, \dots, \bforall {x}{\theta} \varphi_n\big)}{\MM} \\= \mng{\bforall {x}{\theta} Q_mz_m\dots Q_1 z_1 f( \varphi_1, \dots, \varphi_n)}{\MM}.
    \end{equation*}
  \end{thm}
  \begin{proof}
    We prove the statement by induction on the number $m$ of
    (nonbounded) quantifiers. If $m=0$, we have the
    statement by Proposition~\ref{prop-C}.  Now assume that we have
    the statement for some $m=k$, and prove that we have it for $m=k+1$.

    There are two cases:
    \begin{enumerate}
      \item[1.)] either $Q_{k+1}=\exists$, and then 
        \begin{align*}
      &\mng{\forall x \big(\theta \rightarrow \exists  z_{k+1} Q_kz_k\dots Q_1 z_1 f( \varphi_1, \dots, \varphi_n)\big)}{\MM}\\
      & =\mng{\forall x  \exists  z_{k+1} \big(\theta \rightarrow Q_kz_k\dots Q_1 z_1 f( \varphi_1, \dots, \varphi_n)\big)}{\MM} && \text{by \eqref{hm-et} of Remark~\ref{rem-hinman}.}\\
      & =\mng{\exists x \theta \to \exists  z_{k+1} \bforall{x}{\theta} Q_kz_k\dots Q_1 z_1 f( \varphi_1, \dots, \varphi_n)}{\MM} && \text{by Lemma~\ref{prop-exists} and Remark~\ref{rem:2}.}\\
      & =\mng{\exists x \theta \to \exists  z_{k+1}  \big(\exists x\theta \rightarrow Q_kz_k\dots Q_1 z_1 f\big(\bforall{x}{\theta} \varphi_1, \dots\big)}{\MM} && \text{by induction hypothesis.}\\
      & =\mng{\exists x \theta \to \big(\exists x\theta \rightarrow \exists  z_{k+1}  Q_kz_k\dots Q_1 z_1 f\big(\bforall {x}{\theta} \varphi_1, \dots\big)}{\MM} && \text{by \eqref{hm-et} of Remark~\ref{rem-hinman}.}\\
      & =\mng{\exists x \theta \rightarrow  \exists  z_{k+1} Q_kz_k\dots Q_1 z_1 f\big(\bforall {x}{\theta} \varphi_1, \dots\big)}{\MM} && \text{by exportation and idempotency.}
    \end{align*}

    \item[2.)] or either $Q_{k+1}=\forall$, and then 
      \begin{align*}
      &\mng{\forall x \big(\theta \rightarrow \forall  z_{k+1} Q_kz_k\dots Q_1 z_1 f( \varphi_1, \dots, \varphi_n)\big)}{\MM}\\
      & =\mng{\forall x  \forall  z_{k+1} \big(\theta \rightarrow Q_kz_k\dots Q_1 z_1 f( \varphi_1, \dots, \varphi_n)\big)}{\MM} && \text{by \eqref{hm-at} of Remark~\ref{rem-hinman}.}\\
      & =\mng{\forall z_{k+1} \bforall{x}{\theta} Q_kz_k\dots Q_1 z_1 f( \varphi_1, \dots, \varphi_n)}{\MM} && \text{by quantifier interchange and \eqref{eq-bquant}.}\\
      & =\mng{\forall  z_{k+1}  \big(\exists x\theta \rightarrow Q_kz_k\dots Q_1 z_1 f\big(\bforall{x}{\theta} \varphi_1, \dots\big)}{\MM} && \text{by induction hypothesis.}\\
      & =\mng{\exists x\theta \rightarrow \forall  z_{k+1}  Q_kz_k\dots Q_1 z_1 f\big(\bforall {x}{\theta} \varphi_1, \dots\big)}{\MM} && \text{by \eqref{hm-at} of Remark~\ref{rem-hinman}.}
        \end{align*}
    \end{enumerate}
  \end{proof}

    \begin{rem}
      By Proposition~\ref{prop-h} and Remarks~\ref{rem:2} and \ref{rem:boolean}, if the conditions of Theorem~\ref{prop-new} hold and $\MM\models \exists x \theta$, then also the existential quantifiers $\bexists {x}{\theta}$ can be brought out from the corresponding formula, i.e.:
      \begin{equation*}
        \mng{Q_mz_m\dots Q_1 z_1 f\big(\bexists {x}{\theta} \varphi_1, \dots, \bexists {x}{\theta} \varphi_n\big)}{\MM} \\= \mng{\bexists {x}{\theta} Q_mz_m\dots Q_1 z_1 f( \varphi_1, \dots, \varphi_n)}{\MM}.
    \end{equation*} Moreover, any of the quantifiers $\bexists{x}{\theta}$ can freely be replaced by quantifiers $\bforall{x}{\theta}$ in this equation above.
    \end{rem}

\section{Applications}\label{sec:Application}

A main source of applications of these results is simplifying
translations of formulas, where bounded quantifiers appear redundantly
after some translation. Such a situation occurred when special
relativity was interpreted into classical kinematics, see \citep{diss}
and \citep{ClassRelKin}. Here we generalize the simplification rules used there without taking any special restrictions on the formulas $\varphi$, $\iota$ and $\varepsilon$ apart from the variable non-dependence condition introduced in this paper and that the provided condition is of the form $\theta=\iota\land\varepsilon$.
  
For example, in \citep[\S~11~Appendix]{diss}, we define for classical
kinematics that formula $\varphi$ is
\emph{ether-observer-independent} in variable $b$ provided that $k_1$,
\ldots, $k_n$ are inertial observers if the truth or falsehood of
$\varphi$ does not depend on to which ether observer we evaluated $b$:
\begin{equation*}
  EOI^{k_1,\ldots,k_n}_{b}[\varphi] \defiff  \sy{{ClassicalKin}}\vdash\bforall{k_1, \ldots, k_n}{ IOb}\bforall{e_1, e_2}{\Ether} 
  [\varphi(e_1 / b) \leftrightarrow \varphi(e_2 / b)],
\end{equation*}
where $\varphi(e/b)$ means that $b$ gets substituted by $e$ in all free occurrences of $b$ in $\varphi$.

Here, $\bforall{k_1, \ldots, k_n}{ IOb}$ is shorthand using bounded quantifiers for $\forall k_1 \ldots \forall k_n \big( \IOb(k_1) \land \ldots \land \IOb(k_n)\to \dots\big)$, which corresponds to $\iota$ and which asserts that $k_1,\ldots,k_n$ are inertial observers. $\bforall{e_1, e_2}{\Ether}$ is shorthand for $\forall e_1 \forall e_2 \big( \Ether(e_1) \land \Ether(e_2) \to \dots\big)$, which here is $\varepsilon$ and which postulates that $e_1$ and $e_2$ are Ether-observers. So, if we can replace $b$ in $\varphi$ by any ether observer, and $k_1 \ldots k_n$ occuring in $\varphi$ are inertial observers, then $\varphi$ is indeed ether-observer-independent in $b$.

As another example, one of the formulations of the principle of relativity in \citep[Section 4.1]{DiffFormRelPrinc} states that the truth of certain formulas $\varphi(b, \bar x)$ describing experimental scenarios with numerical parameters $\bar x$ does not depend on the choice of inertial observer $b$. This is formulated as an axiom scheme $\mathsf{SPR}^+$ consisting formulas of the form
  \begin{equation*}
    IOb(k)\land IOb(h) \to (\varphi(k,\bar x) \leftrightarrow \varphi(h,\bar x)),
  \end{equation*}
where $\varphi(k,\bar x)$ and $\varphi(h,\bar x)$ are the formula
$\varphi(b,\bar x)$ but variable $b$ is substituted by $k$ and $h$,
respectively.

Let us first connect these notions of independence from both examples above to the non-dependence one introduced in this paper. We will use the following notation for Tarski's substitution:\footnote{This definition of substitution is equivalent to Tarski's definition $\varphi(x/y) \defiff \forall x (x=y \to \varphi)$ in \citep[p.~62]{Tarski1964}, however we use Enderton's notation $\varphi^x_y$ from \citep[p.~112]{End2001} in stead of Tarski's $\varphi(x/y)$. Enderton's definition is equivalent with Tarski's for proper substitution, see \citep[p.~130]{End2001}.}
\begin{equation}\label{eq:substitute}
\varphi^x_y \defiff \exists x (x=y \land \varphi).
\end{equation}

\begin{rem}\label{rem-substitution} Let us note that, by \eqref{eq:substitute} and the definition when $\bar{a}$ satisfies formula $\varphi$ in model $\MM$, we have 
   $\MM\models \varphi^{x}_{v_j}[\bar a]$ if{}f $\MM\models \varphi[\bar  a^{x}_{a_j}]$.
\end{rem}

\begin{prop}\label{prop-reform}
Formula $\varphi$ is non-dependent of $x$ in $\MM$ provided $\theta$
if{}f
\begin{equation}\label{eq-substituted}
  \MM\models (\theta^x_y \land \theta^x_z) \to (\varphi^x_y \leftrightarrow \varphi^x_z)
\end{equation}
for some variable $y$ and $z$ that occur neither in $\varphi$ nor in $\theta$. 
\end{prop}

\begin{proof}
  Let $x=v_i$, $y=v_j$ and $z=v_k$.

  By Remark~\ref{rem-substitution} and the definition of when a
  sequence of elements satisfies a formula in a model,
  \eqref{eq-substituted} is equivalent to that, for all $\bar c \in
  M^\omega$,
  \begin{equation}\label{eq:reformed}
    \MM\models \theta[\bar c^x_{c_j}] \text{ and } \MM\models
    \theta[\bar c^x_{c_k}] \implies (\MM\models \varphi[\bar
      c^x_{c_j}] \iff \MM\models \varphi[\bar c^x_{c_k}]). 
  \end{equation}

  Now assume that $\varphi$ is non-dependent of $x$ in $\MM$ provided
  $\theta$. Then when substituting $\bar{a}=\bar c^x_{c_j}$ and
  $b=c_k$ to \eqref{eq:indep-prov} of Definition~\ref{def:indep-prov},
  we get \eqref{eq:reformed} since $\bar a^x_b=(\bar
  c^x_{c_j})^x_{c_k}=\bar c^x_{c_k}$. This proves the ``$\implies$''
  direction.

  To show the other direction, let $\bar a \in M^\omega$ and $b\in M$
  such as $\MM\models \theta[\bar a]$ and $\MM\models \theta[\bar
    a^x_{b}]$. Let $\bar c$ be the sequence that we get form $\bar a$ by changing the $j$-th element of $\bar a$ to $a_i$ and the $k$-th element of $\bar a$ to $b$, \ie
  \begin{equation*}\bar c \defeq (\bar a^y_{a_i})^z_b = (a_1,\ldots,a_{i-1},\stackrel{i}{a_i},a_{i+1},\dots,a_{j-1},\stackrel{j}{a_i},a_{j+1},\dots,a_{k-1},\stackrel{k}{b},a_{k+1},\dots).
  \end{equation*}  
  Since satisfiability depends only on the evaluations of free
  variables, and variables $y=v_j$ and $z=v_k$ are not free in
  $\theta$ and $\varphi$, and sequences $\bar c$ and $\bar a$ differ
  only in the $j$-th and $k$-th coordinate, we have that $\MM\models
  \theta[\bar a]$ if{}f $\MM\models \theta[\bar c^x_{c_j}]$,
  $\MM\models \varphi[\bar a]$ if{}f $\MM\models \varphi[\bar
    c^x_{c_j}]$, $\MM\models \theta[\bar a^x_b]$ if{}f $\MM\models
 \theta[\bar c^x_{c_k}]$ and $\MM\models \varphi[\bar a^x_b]$ if{}f
  $\MM\models \varphi[\bar c^x_{c_k}]$. Consequently,
  \eqref{eq:reformed} reduces to \eqref{eq:indep-prov}, and hence
  \eqref{eq-substituted} implies Definition~\ref{def:indep-prov}, and
  this is what we wanted to show.
\end{proof}

\begin{figure}
  \begin{center}
    \includegraphics{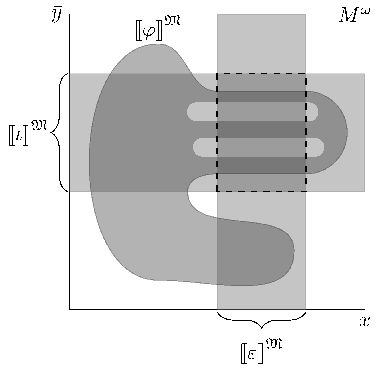}
    \caption{\label{fig-PsiIndepenenceProvidedTheta} \footnotesize{This figure
      illustrates the special case used in
      Section~\ref{sec:Application}, \ie when the provided condition
      is of the form $\theta=\iota\land\varepsilon$ for some formulas
      such that $x$ does not occur free in $\iota$ and certain
        bound variables of $\varphi$ do not occur free in
        $\varepsilon$ and $\iota$. Here $\mng{\theta}{\MM} = \mng{\iota}{\MM} \cap \mng{\varepsilon}{\MM}$ is represented by the area inside the dashed rectangle.}}
  \end{center}
\end{figure}

Now that we have connected the substitution of variables from the above examples to our notion of non-dependence, and we have established the condition $\theta$ as the conjunction $\iota\land\varepsilon$, we can proceed to show how we apply non-dependence to simplify formulas.

\begin{lemma}\label{0HWNp5}
  Let $\varphi$, $\iota$ and $\varepsilon$ be formulas such that variable $x$ does not occur free in $\iota$. Then
  \begin{align*}
   \iota \to\forall x(\iota\land\varepsilon \to \varphi) &\equiv \iota \to \forall x (\varepsilon \to \varphi) \text{ and }\\
   \iota\land\forall x(\iota\land\varepsilon \to \varphi)&\equiv \iota\land\forall x (\varepsilon \to \varphi).
  \end{align*}
\end{lemma}

\begin{proof}
    \begin{align*}
      &\iota \to\forall x(\iota\land\varepsilon \to \varphi)\\
      &\equiv\neg\iota \lor\forall x(\neg\iota\lor\neg\varepsilon \lor \varphi) && \text{by the definition of $\to$ and De Morgan.}\\
      &\equiv\neg\iota \lor\neg\iota\lor\forall x(\neg\varepsilon \lor \varphi) && \text{by \eqref{hm-a} of Remark~\ref{rem-hinman}.}\\
      &\equiv\iota \to \forall x (\varepsilon \to \varphi) && \text{by indempotency and definition of $\to$.}\\
    \end{align*}
    \begin{align*}
      &\iota \land\forall x(\iota\land\varepsilon \to \varphi)\\
      &\equiv\iota \land\forall x(\neg\iota\lor\neg\varepsilon \lor \varphi) && \text{by the definition of $\to$ and De Morgan.}\\
      &\equiv \iota \land \big(\neg\iota\lor\forall x(\neg\varepsilon \lor \varphi)\big) && \text{by \eqref{hm-a} of Remark~\ref{rem-hinman}.}\\
      &\equiv\iota \land \forall x (\varepsilon \to \varphi) && \text{by the distributivity of $\land$ over $\lor$ and identity $(A\land \lnot A)\lor B\equiv B$.}\\
    \end{align*}
\end{proof}

\begin{lemma}\label{HWNp5} Let $f$ be any Boolean expression, let $\iota$, $\varepsilon$ and $\varphi_1,\dots,\varphi_n$ be formulas such that variable $x$ does not occur free in $\iota$. Then
       \begin{align}
         \bforall{u}{\iota}f\big(\forall x(\varepsilon\to\varphi_1),\dots,\forall x(\varepsilon \to\varphi_n)\big)&\equiv\bforall{u}{\iota}f\big(\forall x(\iota\land\varepsilon\to\varphi_1),\dots,\forall x(\iota\land\varepsilon\to\varphi_n)\big),\\
         \bexists{u}{\iota}f\big(\forall x(\varepsilon\to\varphi_1),\dots,\forall x(\varepsilon \to\varphi_n)\big)&\equiv\bexists{u}{\iota}f\big(\forall x(\iota\land\varepsilon\to\varphi_1),\dots,\forall x(\iota\land\varepsilon\to\varphi_n)\big).
       \end{align}

\end{lemma}
\begin{proof}
    Since $\bforall{u}{\iota}\varphi$ abbreviates $\forall u (\iota \to \varphi)$ and  $\bexists{u}{\iota}\varphi$ abbreviates $\exists u (\iota \land \varphi)$, it is enough to prove that
    \begin{align}\label{ind-to}
      \iota\to f\big(\forall x(\varepsilon\to\varphi_1),\dots,\forall x(\varepsilon \to\varphi_n)\big)&\equiv \iota \to f\big(\forall x(\iota\land\varepsilon\to\varphi_1),\dots,\forall x(\iota\land\varepsilon\to\varphi_n)\big),\\
   \label{ind-land}
      \iota\land f\big(\forall x(\varepsilon\to\varphi_1),\dots,\forall x(\varepsilon \to\varphi_n)\big)&\equiv\iota\land f\big(\forall x(\iota\land\varepsilon\to\varphi_1),\dots,\forall x(\iota\land\varepsilon\to\varphi_n)\big).
    \end{align}
  We are going to prove this by a parallel induction on the complexity
  of $f(\varphi_1 \dots \varphi_n)$. By Lemma~\ref{0HWNp5}, we have
  the statements \eqref{ind-to} and \eqref{ind-land} for each
  $\varphi_i$.

    Let us first assume that $f$ is of the form $f = g\land h$, and we
    already know the statements \eqref{ind-to} and \eqref{ind-land}
    for $g$ and $h$, \ie
    \begin{align}\label{ind2-g}
            \iota\to g\big(\forall x(\varepsilon\to\varphi_1),\dots\big)&\equiv \iota \to g\big(\forall x(\iota\land\varepsilon\to\varphi_1),\dots\big)\\
            \iota\land g\big(\forall x(\varepsilon\to\varphi_1),\dots\big)&\equiv \iota\land g\big(\forall x(\iota\land\varepsilon\to\varphi_1),\dots\big)\label{ind2e-g}
      \end{align}
   \begin{align}\label{ind2-h}
            \iota\to h\big(\forall x(\varepsilon\to\varphi_1),\dots\big)&\equiv\iota \to h\big(\forall x(\iota\land\varepsilon\to\varphi_1),\dots\big)\\
            \iota\land h\big(\forall x(\varepsilon\to\varphi_1),\dots\big)&\equiv \iota\land h\big(\forall x(\iota\land\varepsilon\to\varphi_1),\dots\big)\label{ind2e-h}.
      \end{align}
   Then we have \eqref{ind-to} for $f$ because of the following.
   \begin{align*}
      &\iota \to \big(g(\forall x(\varepsilon\to\varphi_1),\dots)\land h(\forall x(\varepsilon\to\varphi_1),\dots)\big)   \\
     &\equiv\big(\iota \to g(\forall x(\varepsilon\to\varphi_1),\dots)\big)\land\big(\iota \to h(\forall x(\varepsilon\to\varphi_1),\dots)\big) && \text{by the distributivity of $\to$ over $\land$.}\\
     &\equiv\big(\iota \to g(\forall x(\iota\land\varepsilon\to\varphi_1),\dots)\big)\land\big(\iota \to h(\dots)\big) && \text{by hypotheses \eqref{ind2-g} and \eqref{ind2-h}.}\\
      &\equiv\iota \to \big(g(\forall x(\iota\land\varepsilon\to\varphi_1),\dots)\land h(\forall x(\iota\land\varepsilon\to\varphi_1),\dots)\big)  &&  \text{by the distributivity of $\to$ over $\land$.}
   \end{align*}
    And we have \eqref{ind-land} for $f$ because of the following.
      \begin{align*}
       &\iota \land \big(g(\forall x(\varepsilon\to\varphi_1),\dots)\land h(\forall x(\varepsilon\to\varphi_1),\dots)\big)   \\
      &\equiv\big(\iota \land g(\forall x(\varepsilon\to\varphi_1),\dots)\big)\land\big(\iota \land h(\forall x(\varepsilon\to\varphi_1),\dots)\big) && \text{by idempotency and associativity.}\\
      &\equiv\big(\iota \land g(\forall x(\iota\land\varepsilon\to\varphi_1),\dots)\big)\land\big(\iota \land h(\dots)\big) && \text{by hypotheses \eqref{ind2e-g} and \eqref{ind2e-h}.}\\
          &\equiv\iota \land \big(g(\forall x(\iota\land\varepsilon\to\varphi_1),\dots)\land h(\forall x(\iota\land\varepsilon\to\varphi_1),\dots)\big)  &&  \text{by idempotency and associativity.}
      \end{align*}

    Let us now assume that $f$ is of the form $f = \neg g$, and we
    already know the statements for $g$. Then  we have \eqref{ind-to} for $f$ because of the following.
    \begin{align*}
     &\iota\to \neg g(\forall x(\varepsilon\to\varphi_1),\dots)   \\
     & \equiv\neg \iota \lor \neg
        g(\forall x(\varepsilon\to\varphi_1),\dots) && \text{by the definition of $\to$.}\\
      & \equiv\neg \big(\iota \land g(\forall x(\varepsilon\to\varphi_1),\dots)\big)  && \text{by De Morgan.}\\
     & \equiv\neg\big(\iota\land g(\forall x(\iota\land\varepsilon\to\varphi_1),\dots)\big)  && \text{by hypothesis \eqref{ind2e-g}.}\\
     & \equiv\iota\to\neg g(\forall x(\iota\land\varepsilon\to\varphi_1),\dots)  &&  \text{by De Morgan and definition of $\to$.}
    \end{align*}
    And we have \eqref{ind-land} for $f$ because of the following.
    \begin{align*}
     &\iota\land\neg g(\forall x(\varepsilon\to\varphi_1),\dots)   \\
      & \equiv\neg \big(\neg\iota \lor g(\forall x(\varepsilon\to\varphi_1),\dots)\big)  && \text{by double negation and De Morgan.}\\
            &\equiv\neg\big(\iota \to g(\forall x(\varepsilon\to\varphi_1),\dots)\big)  && \text{by the definition of $\to$.}\\
     & \equiv\neg\big(\iota\to g(\forall x(\iota\land\varepsilon\to\varphi_1),\dots)\big)  && \text{by hypothesis \eqref{ind2-g}.}\\
      & \equiv\neg\big(\neg\iota\lor g(\forall x(\iota\land\varepsilon\to\varphi_1),\dots)\big) &&  \text{by the definition of $\to$.}
      \\
     & \equiv\iota\land\neg g(\forall x(\iota\land\varepsilon\to\varphi_1),\dots)  &&  \text{by De Morgan and double negation.}
    \end{align*}
Since we have proven this for $\land$ and $\neg$, it follows from Remark~\ref{rem-abbr} that we have proven this for all logical connectives.
\end{proof}

\begin{thm}\label{prop-simp} Let $\MM$ be a model, let $f$ be any Boolean expression, let $\iota$ and $\varepsilon$ be formulas such that variable $x$ does not occur free in $\iota$, and let $\varphi_1,\dots,\varphi_n$ be formulas such that each of $\varphi_1,\dots,\varphi_n$ is non-dependent of variable $x$ in $\MM$ provided $\iota \land\varepsilon$ and $\MM\models \exists x \varepsilon$, and let $Q_1,\dots, Q_k$ as well as $\bar Q$ be arbitrary series of universal and existential quantifiers\footnote{We only care about the individual quantifiers $Q_1,\dots, Q_k$. Since the quantifiers in $\bar Q$ are never refered to individually, we do not need to number them (but we could have numbered them, say as $Q_{k+1},\dots, Q_{k+m}$).}, then
  \begin{multline*}
      \mng{(Q_1u_1\in \iota)\dots (Q_ku_k\in\iota)\bar Q \bar zf\big(\bforall{x}{\varepsilon}(\varphi_1),\dots,\bforall{x}{\varepsilon}(\varphi_n)\big)}{\MM}\\=\mng{(Q_1u_1\in \iota)\dots (Q_ku_k\in\iota)\bforall{x}{\varepsilon}\bar Q \bar zf\big(\varphi_1,\dots,\varphi_n\big)}{\MM}
  \end{multline*}
if no variables of $\bar z$ occur free in $\iota$ and $\varepsilon$.
\end{thm}

\begin{proof}
Since $x$ does not occur free in $\iota$ and $\MM\models  \exists x \varepsilon$, by Remarks~\ref{rem-hinman} and \ref{rem:mng-nice}, we have
\begin{equation} \label{eq-C0}
\mng{\exists x (\iota\land \varepsilon)\to \psi}{\MM}
= \mng{(\iota\land \exists x \varepsilon)\to \psi}{\MM}
= \mng{\iota \to \psi}{\MM}
\end{equation}
for any formula $\psi$. Similarly,
\begin{equation}\label{eq-C1}
\mng{\iota \land \big(\exists x (\iota\land \varepsilon)\to \psi\big)}{\MM} = \mng{\iota \land \psi}{\MM}
\end{equation} because
\begin{align*}
  \mng{\iota \land \big(\exists x (\iota\land \varepsilon)\to \psi\big)}{\MM}
  &= \mng{\iota \land \big((\iota\land \exists x \varepsilon)\to \psi\big)}{\MM}
  && \text{since $x$ is not free in $\iota$.} \\
  &= \mng{\iota \land (\iota \to \psi)}{\MM}
  && \text{since $\exists x \varepsilon$ is true in $\MM$.} \\  
  &= \mng{\iota \land \psi}{\MM} && \text{by identity $A \land (A \to B)\equiv A\land B$}.
\end{align*}

By Theorem~\ref{prop-new} applied on $\theta=\iota\land\varepsilon$ and the definition \eqref{eq-bquant} of bounded quantifiers, we get the following:
  \begin{multline}\label{eq-C}
      \mng{\exists x \big(\iota\land \varepsilon\big)\to \bar Q \bar zf\big(\forall x (\iota\land\varepsilon\to\varphi_1),\dots,\forall x (\iota\land\varepsilon\to\varphi_n)\big)}{\MM} \\
      =\mng{\forall x \big(\iota\land\varepsilon \to \bar Q \bar zf(\varphi_1,\dots,\varphi_n)\big)}{\MM}.
  \end{multline}   
  
  If $Q_k=\forall$, we get the statement as follows:
    \begin{align*}
      &\mng{\bforall{u_k}{\iota} \bar Q \bar zf\big(\bforall{x}{\varepsilon}(\varphi_1),\dots,\bforall{x}{\varepsilon}(\varphi_n)\big)}{\MM} \\
      &=\mng{\forall u_k\Big(\iota\to \bar Q \bar zf\big(\forall x (\iota\land\varepsilon\to\varphi_1),\dots)\big)\Big)}{\MM} && \text{by \eqref{eq-bquant} and Lemma~\ref{HWNp5}.} \\
            &=\mng{\forall u_k\Big(\exists x (\iota\land\varepsilon)\to \bar Q \bar zf\big(\forall x (\iota\land\varepsilon\to\varphi_1),\dots\big)\Big)}{\MM}   
            &&\text{by equation \eqref{eq-C0}.}\\
      &=\mng{\forall u_k\forall x \big(\iota\land\varepsilon\to \bar Q \bar zf(\varphi_1,\dots,\varphi_n)\big)}{\MM}   &&\text{by equation \eqref{eq-C}}.\\
      &=\mng{\forall{u_k}\forall x \big(\iota\to(\varepsilon\to \bar Q \bar zf(\varphi_1,\dots,\varphi_n))\big)}{\MM}   &&\text{by exportation.}\\
      &=\mng{\forall{u_k} \big(\iota\to\forall x(\varepsilon\to \bar Q \bar zf(\varphi_1,\dots,\varphi_n))\big)}{\MM}   &&\text{by \eqref{hm-at} of Remark~\ref{rem-hinman}.}\\
           &=\mng{\bforall{u_k}{\iota}\bforall{x}{\varepsilon}\bar Q \bar zf(\varphi_1,\dots,\varphi_n)}{\MM}   &&\text{by \eqref{eq-bquant}.}
    \end{align*}
    If $Q_k=\exists$, we get the statement as follows:
    \begin{align*}
      &\mng{\bexists{u_k}{\iota} \bar Q \bar zf\big(\bforall{x}{\varepsilon}(\varphi_1),\dots,\bforall{x}{\varepsilon}(\varphi_n)\big)}{\MM} \\
    &=\mng{\exists u_k\Big(\iota \land \bar Q \bar zf\big(\forall x (\iota\land\varepsilon\to\varphi_1),\dots\big)\Big)}{\MM}    &&\text{by \eqref{eq-bquant} and Lemma~\ref{HWNp5}}.\\
      &=\mng{\exists u_k\Big(\iota \land\big(\exists x (\iota \land \varepsilon) \to \bar Q \bar zf\big(\forall x (\iota\land\varepsilon\to\varphi_1),\dots\big)\big)\Big)}{\MM}   
      &&\text{by equation \eqref{eq-C1}.}\\
            &=\mng{\exists u_k\Big(\iota \land \forall x \big(\iota\land\varepsilon\to \bar Q \bar zf(\varphi_1,\dots,\varphi_n)\big)\Big)}{\MM}   &&\text{by equation \eqref{eq-C}}.\\
      &=\mng{\exists{u_k}\big(\iota\land\forall x \big(\iota\to(\varepsilon\to \bar Q \bar zf(\varphi_1,\dots,\varphi_n))\big)\big)}{\MM}   &&\text{by exportation.}\\
      &=\mng{\exists{u_k}\big(\iota\land \big(\iota\to\forall x(\varepsilon\to \bar Q \bar zf(\varphi_1,\dots,\varphi_n))\big)\big)}{\MM}   &&\text{by \eqref{hm-at} of Remark~\ref{rem-hinman}.}\\
            &=\mng{\exists{u_k}\big(\iota\land \big(\neg\iota  \lor \forall x(\varepsilon\to \bar Q \bar zf(\varphi_1,\dots,\varphi_n))\big)\big)}{\MM} &&\text{by  definition of $\to$.}\\
      &=\mng{\exists{u_k}\big((\iota\land \neg\iota) \lor \big(\iota\land  \forall x(\varepsilon\to \bar Q \bar zf(\varphi_1,\dots))\big)\big)}{\MM}   &&\text{by distributivity of $\land$ over $\lor$.}\\
      &=\mng{\exists{u_k}\big(\big(\iota\land\forall x(\varepsilon\to \bar Q \bar zf(\varphi_1,\dots,\varphi_n))\big)\big)}{\MM}   &&\text{by identity $(A\land \lnot A)\lor B\equiv B$.}\\
        &=\mng{\bexists{u_k}{\iota}\bforall{x}{\varepsilon}\bar Q \bar zf(\varphi_1,\dots,\varphi_n)}{\MM}   &&\text{by idempotency and \eqref{eq-bquant}.}
    \end{align*}
  \end{proof}

Instead of the Lemmas of \citep[\S~11~Appendix]{diss}, Theorem~\ref{prop-simp} provides a generic alternative for simplifying translations of formulas to their desired form in the interpretations used in \citep{diss} and \citep{ClassRelKin}. 

In relation to the $\mathsf{SPR}^+$ formulation of the principle of relativity from \citep[Section 4.1]{DiffFormRelPrinc}, our approach gives an alternative point of view, namely understanding the principle of relativity as a simple variable non-dependence of certain formulas describing experiments. By Proposition~\ref{prop-reform}, in terms of variable non-dependence, $\mathsf{SPR^+}$ basically states that any formula $\varphi$ describing an experimental scenario for $x$ with some numerical parameters $\bar y$ (assuming all the free variables of $\varphi$ are among $x$ and elements of $\bar y$) is non-dependent of variable $x$ provided $x$ is an inertial observer.

We believe the above results can be useful in other situations where automatically generated formulas need to be cleaned up, as well as for the developments of algorithms for simplifying formulas.

\section*{Acknowledgments}
We are grateful to Hajnal Andr{\'e}ka, Mich{\`e}le Friend, Zal{\'a}n Gyenis, Istv{\'a}n N{\'e}meti, and Jean Paul Van Bendegem for enjoyable discussions and feedback while writing this paper.
This research was supported by the Hungarian National  Research, Development and Innovation Office (NKFIH), grants  no.\ FK-134732 and TKP2021-NVA-16.


\bibliographystyle{agsm}
\bibliography{LogRel12017}

\pagebreak
\[ \]\footnotesize
\begin{flushright}
KOEN LEFEVER\\
Centre for Logic and Philosophy of Science/Centrum Leo Apostel, Vrije Universiteit Brussel\\
\& Belgian Science Policy Office\\
\href{mailto:koen.lefever@vub.be}{koen.lefever@vub.be}\\
\url{http://lefever.space/}

\vspace{.7cm}

GERGELY SZ{\' E}KELY\\
HUN-REN Alfr{\' e}d R{\' e}nyi Institute for Mathematics\\
\& University of Public Service, Budapest, Hungary\\
\href{mailto:szekely.gergely@renyi.hu}{szekely.gergely@renyi.hu}\\
\url{http://www.renyi.hu/~turms/}
\end{flushright}

\end{document}